\documentclass[review]{elsarticle}
\usepackage[]{latexsym,amssymb,amsmath,color,subfigure,amsthm,}
\usepackage[pdftex]{hyperref}
\usepackage{graphicx}
\usepackage{diagbox}
\usepackage{verbatim}
\usepackage{bm}
\allowdisplaybreaks
\usepackage{caption}

\newtheorem{remark}{Remark}[section]
\newtheorem{prop}{Proposition}[section]
\newtheorem{corollary}{Corollary}[section]
\newtheorem{theorem}{Theorem}[section]
\newtheorem{lemma}{Lemma}[section]
\newtheorem{assumption}{Assumption}[section]
\newtheorem{defn}{Definition}[section]

\newcommand{\jump}[1]{\ensuremath{\left[\!\!\left[#1\right]\!\!\right]} }

\newcommand{\sign}{\text{\rm sign}}
\newcommand{\Hs}{\mathcal{H}}

\begin{document}

\begin{frontmatter}

\title{Fracture Model Reduction and Optimization for Forchheimer Flows in Reservoir}

\author{P.J.~Paranamana}
\ead{pushpi.paranamana@ttu.edu}
\address{Department of Mathematics and Statistics, Texas Tech University, Lubbock, TX, 79409-1042}
\author{E.~Aulisa\corref{cor1}}
\ead{eugenio.aulisa@ttu.edu}
\address{Department of Mathematics and Statistics, Texas Tech University, Lubbock, TX, 79409-1042}
\author{A.~Ibragimov\corref{}}
\ead{akif.ibraguimov@ttu.edu}
\address{Department of Mathematics and Statistics, Texas Tech University, Lubbock, TX, 79409-1042}
\author{M.~Toda\corref{}}
\ead{magda.toda@ttu.edu}
\address{Department of Mathematics and Statistics, Texas Tech University, Lubbock, TX, 79409-1042}

\cortext[cor1]{Corresponding author}

\begin{abstract}
In this study, we analyze the flow filtration process of slightly compressible fluids in fractured porous media. We model the coupled fractured porous media system, where the linear Darcy flow is considered in porous media and the nonlinear Forchheimer equation is used inside the fracture. 

Flow in the fracture is modeled as a reduced low dimensional BVP which is coupled with an equation in the reservoir. We prove that the solution of the reduced model can serve very accurately to approximate the solution of the actual high-dimensional flow in reservoir fracture system, because the thickness of the fracture is small. In the analysis we consider two types of Forchhemer flows in the fracture: 
isotropic and anisotropic, which are different in their nature.

Using method of reduction, we developed a formulation for an optimal design of the fracture, which  maximizes the capacity of the fracture in the reservoir with fixed geometry. Our method, which is based on a set point control algorithm, explores the coupled impact of the fracture geometry and $\beta-$Forchheimer coefficient. 
\end{abstract}


\end{frontmatter}
	
\author{Pushpi J. Paranamana, Eugenio Aulisa, Akif Ibragimov, Magdalena Toda}
	




\section{Introduction}

The mathematical modeling and analysis of strongly coupled nonlinear problems in fractured 
porous media represent very valuable tools towards improving the modern methodologies of oil/gas recovery which use fracturing.  Modern technologies allow numerous man-made fractures, in order to increase the productivity of the reservoir. 
Identifying an optimal fracture configuration for a given reservoir geometry is the major task in the fracturing technology. 
Due to the complexity of the parameters involved (such as the geometry of the reservoir, geometry of the fracture, and the velocity and the pressure of the nonlinear flow), finding an optimal length of the fracture that maximizes productivity is challenging. 

Due to the fractures in the porous block, the velocity of filtration is higher, and therefore 
the inertial forces alter the linear relation between the velocity and the pressure gradient in Darcy's law
\cite{muskat1938flow, forchheimer1901wasserbewegung, dake1983fundamentals, bear2013dynamics}. Therefore, studying the flow with non-Darcy effects is necessary \cite{whitaker1996forchheimer, balhoff2007predictive, miskimins2005non, polubarinova2015theory, ruth1992derivation}.
The Darcy-Forchheimer equation (which demonstrates the momentum conservation), the continuity equation and the thermodynamical equation of the state between fluid density and pressure represent the building blocks towards modeling the flow in the reservoir domain \cite{aulisa2009mathematical, dake1983fundamentals, muskat1938flow}. We use a generalized nonlinear Forchheimer equation that serves as the momentum equation, in order to describe the hydrodynamics in fractured porous media. \cite{payne1999convergence, payne1999continuous, payne1997spatial, payne1996stability}. 

In this work, we discuss slightly compressible fluids with small permeability coefficients where the dissipative effect is dominant, as in natural reservoirs \cite{muskat1938flow}.
Since our applications are in geophysics and reservoir engineering, we consider fluid flows in the domain with different regimes of filtration in fracture and porous blocks, bounded by exterior impermeable boundary and interior  boundary representing well.

To analyze the notion of well productivity index (PI), that is used in reservoir engineering to measure the capacity of the well and the available reserves \cite{dake1983fundamentals, raghavan1993well}, we introduce an integral functional called Diffusive Capacity defined as total flux on the well surface divided by the pressure drawdown (difference between average pressures of the reservoir domain and the well boundary). Diffusive capacity mathematically is equal to PI.

Diffusive capacity is time independent in general. 
However, regardless of initial conditions, the Diffusive Capacity stabilizes to a constant value in long term dynamics \cite{muskat1938flow, raghavan1993well}. We analyze  the regime with constant well production $Q$, which in general gives transient diffusive capacity. Even though the diffusive capacity depends on the initial data, observations from reservoir engineering show that, this time dependent characteristic 
stabilizes to  the diffusive capacity of the PSS regime which in a sense serves as an attractor for the transient one \cite{dake1983fundamentals, raghavan1993well}. In case of PSS regime, pressure distribution can be represented as $P_{PSS(x,t)}=-At+W(x),$ where $W(x)$ is a solution of specific steady state boundary value problem (BVP), and corresponding steady state diffusive capacity uniquely can be calculated  as $\frac{Q}{\int_\Omega W(x)dx}$, where $\Omega$ is the size of the reservoir. In this article, we analyze the impact of the complex geometry of the fracture and non-linearity of the flow, on the value of steady state diffusive capacity. We explore the fractured reservoir domain, where the linear Darcy law is considered in the reservoir and the non linear Forchheimer equation is considered in the fracture.

We consider a reservoir/well system in which the vertical flow is negligible. Therefore, the domain of the flow is essentially two-dimensional. Understanding the behavior of hydrodynamical properties in a simple setting will provide a better insight into understanding the behavior of more complex, higher dimensional fluid structure problems.
We consider the  case when the fracture can be modeled as a thin but long rectangle with thickness $h$ and length $L$. 
For  man-made and natural geological fractures, the thickness ("opening of the fracture") is less than $0.1 mm$, while the characteristic reservoir size can be larger than $500 m$, therefore, solving the original two dimensional problem is numerically almost impossible. At the same time, velocity along the fracture length  is much higher than the velocity inside porous media towards fracture. 

Engineers often in practice, simplify the problem in 2-D modeling of reservoir fracture system and treat the fracture as 1-D sink with pressure equal to the value of the pressure on the well. 
This assumption leads to a significant over estimation of fracture capacity.

The goal of the research is to consider the actual nature of the problem which includes the nature of the flow inside the fracture and couple it with the flow in the porous media.   
We introduce a novel model of 1-D non-linear flow inside fracture and a source term distributed along the length. This source term coupled with the flow in the porous block, explicitly depend on the fracture thickness $h$ and the $\beta$ coefficient of the Forchheimer equation in the fracture. Then we analyze the difference between the solutions of the original problem with 2-D fracture and the one-dimensional one, in order to justify the introduction of the reduced model. We investigate isotropic and anisotropic flows in the fracture, separately.
For the isotropic case for a given fracture, we prove the closeness of the two solutions of the original and the reduced 1-D model. 
In the anisotropic case, we consider the non-linear Forcheimer equation only along the fracture while the flow remains linear in the direction perpendicular to the fracture. For this case we prove a stronger result. Namely, even if the individual solutions are unbounded as the fracture thickness approaches zero, the difference between the solutions of the two problems is bounded, and does not depend on the size of the fracture thickness.

As a conclusion we prove the following Engineering statement.\\
\textit{Since the fracture thickness is several orders of magnitude smaller than the dimension of the reservoir, 
the reduced one dimensional model of the fracture provides an accurate value of the capacity of the actual 2-D fracture, 
and can take into account the Forchheimer coefficient $\beta,$ the thickness and the length of the fracture.   } 

We apply our method to investigate an optimization problem, which is of vital importance in fracturing technology.   
Using the developed model, we explore the optimization problem for the capacity of the fracture, 
with respect to its length, thickness, and Forchheimer coefficient $\beta$. 
As optimization factor, we use the diffusive capacity. The goal is to investigate the optimal fracture length 
that maximizes the diffusive capacity.

We solve the optimization problem in last section of the article, as the following inverse problem. 

\textbf{Inverse Problem:}
\textit{ For a given length of the fracture, a prescribed pressure drawdown ``PDD'' in the reservoir, and a Forchheimer $\beta$ parameter in the fracture, find the value of the total rate of production $Q$ such that the solution of the pseudo-steady state problem has the prescribed value of ``PDD''.}

It is evident that for linear Darcy flow ``PDD'' is linear with $Q$ and therefore the problem is much simpler.
To solve this problem for non-linear flows, we develop a mathematical formulation using a set point control algorithm \cite{aulisa2014numerical, aulisa2015practical}, and implement corresponding numerical schemes. For a fixed pressure drawdown, we find the diffusive capacity for different fracture lengths and various nonlinear flows. Numerical simulations were performed in order to calculate the diffusive capacity of a reservoir with a horizontal fracture - for different reservoir geometries. We analyze how the diffusive capacity changes as the fracture length and nonlinear term of the flow vary.

\section{Formulation of the Problem}
\subsection{Reservoir/Porous Media Modeling}
The pressure function plays a vital role in the oil filtration process in porous media bounded by the reservoir boundary and the well surface. Oil/gas is drained by the well built in the bounded reservoir domain. While small rate of productions, which are characterized by low velocity flows, are associated to the Darcy's law, fluid flows with high production rates give turbulent regimes, which create non Darcy phenomena \cite{forchheimer1901wasserbewegung, muskat1938flow, dake1983fundamentals, bear2013dynamics, balhoff2010polynomial}. Although there are various methods to model the non-Darcy case, the genaral Brinkman-Forchheimer equation appeares to be the most suitable \cite{ewing1999numerical, douglas1993generalized, forchheimer1901wasserbewegung, payne1999convergence}.

Let  $\mathbf{v}$ be the velocity vector field and $p$ be the pressure in porous media. The time dependent  Brinkman-Forchheimer equation is given by \cite{payne1999convergence, aulisa2009mathematical},
\begin{equation}
\left(\rho c_a \frac{\partial \mathbf{v}}{\partial t}- \mu \Delta \mathbf{v} \right)+\frac{\mu}{k}\mathbf{v}+\beta |\mathbf{v}|\mathbf{v}=-\nabla p ,\mbox{ with} \,
\beta=\frac{F\rho \phi}{k^{1/2}} \label{B-F eqn}\,\,, \end{equation}
where $c_a$ is the acceleration coefficient, $F$ is the Forchheimer coefficient, $\phi$ is the porosity, $k$ is the permeability, $\mu$ is the viscosity, and $\rho$ is the density of the fluid.
Eq.\eqref{B-F eqn} is interpreted as the momentum conservation equation. 

\begin{remark}\label{inhom}
In our intended application the parameters of the reservoir and fracture are isotropic and spacial  dependent. Namely, 
$k=k_p$ and $\beta=0$ in the porous block, and $k=k_f$ and $\beta\neq 0$ in the fracture.
\end{remark}

The continuity equation is given by,
\begin{equation}
\frac{\partial \rho}{\partial t}+\nabla \cdot (\rho \mathbf{v})=0\,\,.
\end{equation}

We introduce some constraints (assumptions) and reduce the above equations to a scalar quasi-linear parabolic equation for pressure only.

\begin{assumption}\label{dissipation}
The permeability coefficient $k$ is assumed to be very small, meaning that the dissipation in the porous media is dominant. That is, the first two terms on the L.H.S. of Eq.\eqref{B-F eqn} (inertial and viscous terms) are assumed to be negligible. This is the case for natural reservoirs \cite{muskat1938flow}.
\end{assumption}

\begin{assumption}\label{slightly com}
	The fluid is assumed to be slightly compressible and satisfies the equation of state given by
	\begin{equation}
	\rho^{\prime}=\gamma^{-1}\rho \,\,\,\,\,\,\left(\rho={\rho_0}\exp^{{\gamma^{-1}}(p-p_0)}\right)\,\,,
		\end{equation}
		where $\gamma$ is the compressibility constant of the fluid.
\end{assumption}
From assumptions \ref{dissipation} - \ref{slightly com}, we obtain the following system of equations:
\begin{align}
\rho^{\prime}\frac{\partial p}{\partial t}=-\rho \nabla\cdot\mathbf{v}-\rho^{\prime}\mathbf{v}\cdot\nabla p\,\,,\label{eqn}\\
-\nabla p-\frac{\mu}{k}\mathbf{v}-\beta|\mathbf{v}|\mathbf{v}=0\,\,,\\
\rho^{\prime}=\gamma^{-1}\rho \label{state eqn}\,\,.
\end{align}

\begin{assumption}\label{gamma small}
Since for many slightly compressible liquids $\gamma^{-1}$ is of order $10^{-8}$ \cite{aronson1986porous, dake1983fundamentals}, the term $\rho^{\prime}\mathbf{v}\cdot\nabla p$ in Eq.\eqref{eqn} is negligible.
\end{assumption}
Finally, the system can be rewritten as
\begin{align}
\rho^{\prime}\frac{\partial p}{\partial t}=-\rho \nabla\cdot\mathbf{v}\,\,,\label{continuity eqn}\\
-\nabla p-\frac{\mu}{k}\mathbf{v}-\beta|\mathbf{v}|\mathbf{v}=0 \label{F eqn}\,\,.
\end{align}
We will refer to Eq.\eqref{continuity eqn} as the continuity equation, Eq.\eqref{F eqn} as the Forchheimer equation, and Eq.\eqref{state eqn} as the state equation for the slightly compressible fluids.\\\\
\textbf{Darcy-Forchheimer equation:}
The velocity vector field $\mathbf{v}$ as a dependent variable can be uniquely represented as the following function of the pressure gradient.
\begin{align} 
\mathbf{v} = \mathbf{v}_\beta =  -f_\beta \left
(\|\nabla p\|\right) \nabla p \, , \label{DF eqn} \\
f_\beta\left(\|\nabla p\|\right) = \frac{2}{\alpha+\sqrt {\alpha^2 +4 \beta \|\nabla p\|}} \,. \label{non lin term}
\end{align}
where
$\alpha=\frac{\mu}{k}.$\\
Eq.\eqref{DF eqn} is referred as the Darcy-Forchheimer equation.

Note that in accordance to Remark\ref{inhom} function $f_{\beta}$ is different in porous block and inside fracture.

\begin{lemma} \label{fbeta}
	Let $f_\beta$ be defined by Eq.\eqref{non lin term}. Then, the velocity defined in the Darcy Forchheimer equation \eqref{DF eqn} solves the Forchheimer equation \eqref{F eqn}.
\end{lemma}
\begin{proof} The proof follows from Lemma 2.1 in Ref.\cite{aulisa2009mathematical}.
\end{proof}

\begin{assumption}\label{beta indep}
	The coefficient $\beta$ in the non linear term of Eq.\eqref{DF eqn} does not depend on pressure. Changes of density for slightly compressible fluids have a minor impact on changes in the coefficient $\beta$ \cite{dake1983fundamentals}.
\end{assumption}

\begin{corollary}
As $\beta \to 0$ in lemma \ref{fbeta}, the Darcy-Forchheimer equation reduces to the Darcy equation.
\end{corollary}

\begin{lemma}
	Let assumptions \ref{dissipation} - \ref{beta indep} hold and $f_\beta$ be defined by the formula \eqref{DF eqn}. Then the pressure function $p$ satisfies the quasi linear parabolic equation
	\begin{equation}
	\frac{\partial p}{\partial t}=\gamma \nabla \cdot \left(f_\beta(\|\nabla p\|)\nabla p\right) \label{parabolic eqn}.
	\end{equation}
\end{lemma}

\begin{proof}
Substituting Eqs.~\eqref{state eqn} and \eqref{DF eqn} in the continuity equation \eqref{continuity eqn}, Eq.~\eqref{parabolic eqn} will be obtained.
\end{proof}

\subsection{Diffusive Capacity and Pseudo-Steady State Regime (PSS)}
The system \eqref{continuity eqn} - \eqref{F eqn} characterizes the fluid flow of well exploitation in a reservoir. Various production regimes can be determined by different boundary conditions \cite{ibragimov2005mathematical}. In our modeling, no flux is coming into or out of the exterior boundary of the reservoir. The IBVP for the system~\eqref{continuity eqn} - \eqref{DF eqn} is formulated by the well boundary condition, impermeable exterior boundary condition, and the initial well pressure. The solution of the IBVP illustrates the hydrodynamical properties of the reservoir well system.

The {\it productivity index} (PI) represents a concept used in reservoir engineering
to characterize the efficiency of the well and the available reserves. It represents the relationship between production rate, average pressure of the reservoir, and the average pressure of the wellbore. To view this mathematically, we introduce the notion of an integral functional defined over the solutions of the IBVP for the system~\eqref{continuity eqn} - \eqref{DF eqn}, called the diffusive capacity \cite{aulisa2009mathematical}. In this section, we explore the diffusive capacity as a generalization of the productivity index. 

Let $\Omega$,  denote the reservoir domain bounded by the exterior no flux boundary$\Gamma_{out}$ and the well surface  $\Gamma_w$,with given total flax.  Flow in vertical direction  assumed to be equal zero.

\begin{figure}[h]
	\begin{center}
		\includegraphics[scale=0.2]{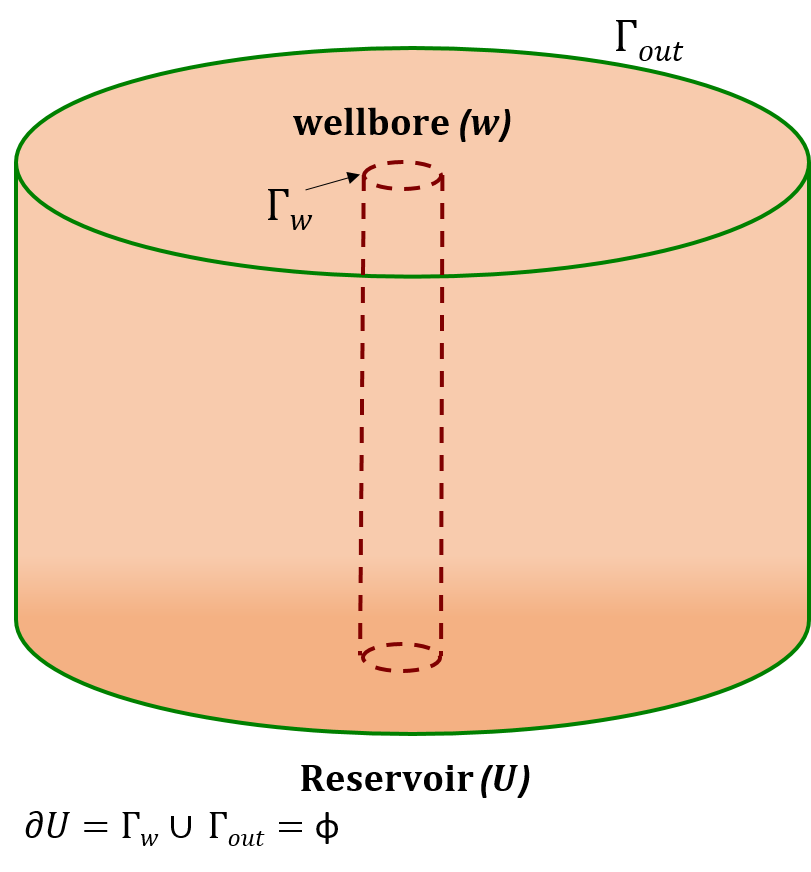}
		\caption{Reservoir Domain}	
	\end{center}
\end{figure}

Let $p_0(x)$ be the initial pressure of the reservoir.
Let $\bar p_{\Omega}(t)=\frac{1}{\|\Omega\|}\int_\Omega  p \ dx$ be the average pressure on $\Omega$ and 
$\bar p_w(t)=\frac{1}{\|\Gamma_w\|} \int_{\Gamma_w} p \ ds $ be the average pressure on $\Gamma_w$
where $\|\Omega\|$ is the volume of the reservoir and
$\|\Gamma_w\|$is the area of the well.

\begin{defn}Diffusive Capacity.\\
Let the pressure $p$ and the velocity $\mathbf{v}$ form the solution for the system~\eqref{continuity eqn} - \eqref{F eqn} in the bounded domain $\Omega$, with impermeable boundary condition $\mathbf{v}\cdot \mathbf{n} \, \big \rvert_{\Gamma_{out}} = 0$ where $\mathbf{n}$ is the outward unit normal on the piecewise smooth surface $\Gamma_{out}$. Assume for any $t>0$, the pressure drawdown (PDD) on the well is positive: $\left(\bar p_\Omega(t)-\bar p_w(t)>0 \right)$. Then the diffusive capacity is defined by
\begin{equation}
J_p(t)=\frac{\int_{\Gamma_w}\mathbf{v}\cdot \mathbf{n}\, ds }{\bar p_\Omega(t)- \bar p_w(t)}.\label{DC} 
\end{equation}
	
\end{defn}

It has been observed on filed Data that when the well production rate $Q(t)= Q= const$, the well productivity index stabilizes to a constant value over time (see fo example \cite{raghavan1993well}).

\begin{defn} PSS regime.\\
	Let the well production rate Q be time independent: $\int_{\Gamma_w} \mathbf{v}\cdot \mathbf{n}\, ds= Q.$ The flow regime is called a pseudo-steady state (PSS) regime, if the pressure drawdown is constant (i.e.,  $\bar p_\Omega(t)-\bar p_w(t)=$constant.)
	\end{defn}

\begin{corollary}
For any PSS regime, the diffusive capacity/PI is time-invariant.
\end{corollary}

\subsubsection{PSS solution for the initial boundary value problem.}
Let $\Omega$ be a bounded reservoir domain with impermeable exterior boundary and inner boundary$\Gamma_w$ which model well. Assume that the well is operating under time independent constant rate of production Q and the initial reservoir pressure is known. Let the assumptions \ref{dissipation} - \ref{gamma small} hold. Then the IBVP modeling the oil filtration process can be formulated as 
\begin{align}
\rho{\prime}\frac{\partial p}{\partial t}=-\rho \nabla\cdot\mathbf{v} \,,\\
-\nabla p-\frac{\mu}{k}\mathbf{v}-\beta|\mathbf{v}|\mathbf{v}=0\,,\\
\int_{\Gamma_w} \mathbf{v}\cdot \mathbf{n} ds= Q\,,\\
\mathbf{v}\cdot \mathbf{n} \, \big \rvert_{\Gamma_{out}} = 0\,,\\
p(x,t_0)= p_0(x)\,.
\end{align}
Using the assumption \ref{beta indep} , the IBVP reduces to \\
\begin{align}
\gamma^{-1}\frac{\partial p}{\partial t}= \nabla \cdot \left(f_\beta(\|\nabla p\|)\nabla p\right)\label{parab eqn}\,,\\
\int_{\Gamma_w} f_\beta \left(\|\nabla p\|\right)\frac{\partial p}{\partial \mathbf{n}}  ds=-Q \label{total flux}\,,\\
\frac{\partial p}{\partial \mathbf{n}}\Big \rvert_{\Gamma_{out}} = 0\,,\\
p(x,t_0)= p_0(x) \label{ini pressure}\,.
\end{align}

Since the boundary condition on $\Gamma_w$ is a single integral condition for the total flux, the IBVP \eqref{parab eqn}-\eqref{ini pressure} has infinitely many solutions and thus, it is ill posed. The diffusive capacity is defined as an integral characteristic of the solution  and therefore lacks the uniqueness. So, we constrain the solution to an auxiliary problem with uniqueness up to an additive constant \cite{aulisa2009mathematical}.

\subsubsection{Steady state auxiliary BVP}
Let $Q$ be the rate of production. Assume that the boundary $\partial \Omega$ is smooth. Let $W$ be the solution of the auxiliary steady state BVP\\
\begin{align}
-\frac{Q}{\|\Omega\|}=\nabla \cdot \left(f_\beta(\|\nabla W\|)\nabla W \right)\,\ \text{in} \ \Omega,\label{aux prob}\\
W \big \rvert_{\Gamma_w}=0\,,\\
\frac{\partial W}{\partial \mathbf{n}}\Big \rvert_{\Gamma_{out}} = 0 \label{aux ex boundary}\,.
\end{align}

Through integration by parts, we obtain
\begin{equation}
\int_{\Gamma_w} f_\beta(\|\nabla W\|)\frac{\partial W}{\partial \mathbf{n}}  ds=-Q\,.
\end{equation}

\begin{prop} \label{soln to aux prob}
	Let $W(x)$ be the solution of the auxiliary problem \eqref{aux prob} - \eqref{aux ex boundary}, then
	\begin{equation}
	p(x,t)=W(x)-\gamma A t + K \label{soln}\,,
	\end{equation}
	where $ A= \frac{Q}{\|\Omega\|}$, solves the IBVP \eqref{parab eqn} - \eqref{ini pressure}. For this solution, we have the pseudo-steady state, and therefore the Diffusive Capacity is constant.
\end{prop}
\begin{proof}
(See Ref. \cite{aulisa2009mathematical}).
Proposition \ref{soln to aux prob} is followed by substituting \eqref{soln} to Eq.\eqref{parab eqn} and verifying boundary and initial conditions \eqref{total flux} - \eqref{ini pressure}.
\end{proof}

\subsection{Fractured Reservoir Modeling}
In this section, we model the oil filtration process with a constant rate of production $Q$, for a fractured reservoir system.

Consider a fractured reservoir domain where the exterior boundary  $\Gamma_{out}$ is impermeable.
Let  $\Omega=\Omega_p\cup \Omega_f$, where $\Omega_p$ be the porous media domain, $\Omega_f$ be the fracture domain, $\Gamma_f$ be the boundary between the fracture, and the porous media and $\Gamma_{f_{out}}$ be the extreme of the fracture.

\begin{figure}[h]
	\begin{center}
		\includegraphics[scale=0.4]{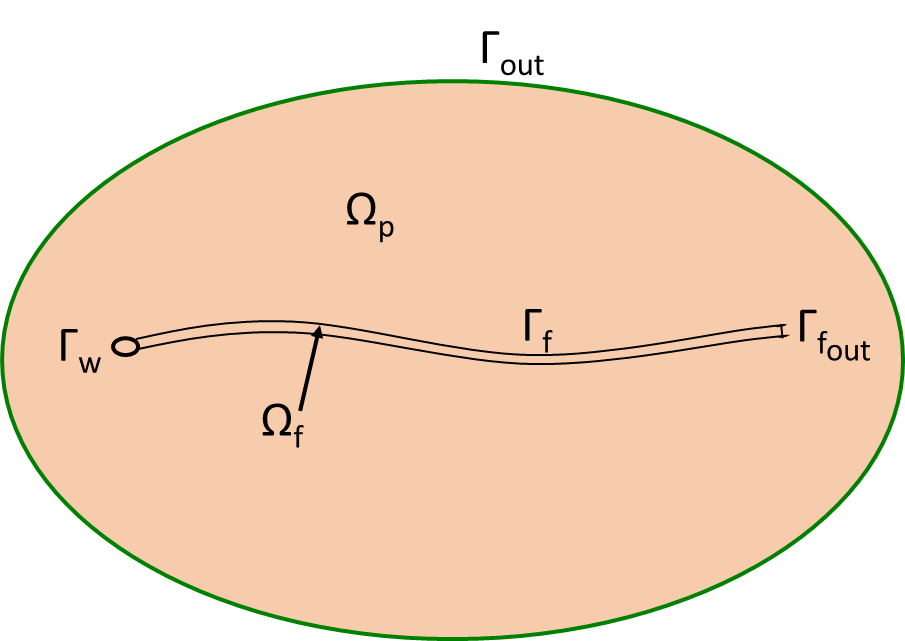}
		\caption{Fractured-Reservoir Domain}	
	\end{center}
\end{figure}
Let $\mathbf{v}_p, W_p, k_p$ be the velocity, the pressure and the permeability of the porous media, respectively. Let $\mathbf{v}_f, W_f$ be the velocity and the pressure of the fracture, respectively.
Let $\mathbf{n}_p$, and $\mathbf{n}_f$ be the unit outward normal vectors 
to the porous medium and the fracture, respectively.

The auxiliary BVP \eqref{aux prob} - \eqref{aux ex boundary} for the above mentioned reservoir-fracture 
system can be modeled as 

\begin{align}
-\frac{Q}{\|\Omega\|} &= \nabla \cdot k_p \nabla W_p & \mbox{ in } \Omega_p \label{porous pressure}\,,\\
-\frac{Q}{\|\Omega\|} &=\nabla \cdot f_\beta(\|\nabla W_f\|) \nabla W_f & \mbox{ in } \Omega_f \label{frac pressure}\,,\\
\mathbf{v}_p \cdot \mathbf{n}_p &= 0  & \mbox{ on } \Gamma_{out} \, ,\\
W_p &= W_f  & \mbox{ on } \Gamma_f \cup \Gamma_{f_{out}} \label{mass_cont}\,,\\
\mathbf{v}_p \cdot \mathbf{n}_p &= - \mathbf{v}_f \cdot \mathbf{n}_f  & \mbox{ on } \Gamma_f \cup \Gamma_{f_{out}}\, ,\label{flux_cont}\\
W &= 0  & \mbox{ on } \Gamma_w \label{ini well} \,.
\end{align}

Here  Eqs.~\eqref{mass_cont} and \eqref{flux_cont} assure the continuity of the solutions and continuity of the fluxes
across the interface $\Gamma_f$. Notice that we consider the linear Darcy law in the reservoir, 
while the non-linear Forchheimer equation is considered in the fracture.
%

\section{Fracture Model Reduction}
Now, we consider the fracture to be the rectangular region $\Omega_f=[0,L]\times[-\frac{h}{2},\frac{h}{2}]$.
In the following, 
we assume the fracture thickness $h$ to be several orders of magnitude smaller
than the dimension of the reservoir. 
Under this assumption, we assume that the fracture $\Omega_f$ is an interior boundary of the flow in domain $\Omega_p$, where a certain PDE
has to be satisfied. 
Similarly, we model the well of the reservoir to be just a point located at the left extreme of the fracture line segment.

Then, we reduce the fracture region to the line segment on the $x$-axis, $\Gamma_f = [0,L]$,
and the well as a point located at the origin. 

For convenience in the reduced model we use the same notation
$\Gamma_f$ for the domain of fracture, that was used in the original model as the boundary of $\Omega_f$.

\begin{defn}
We define
$$\bar{W}_f(x)=\frac{1}{h} \int_{-\frac{h}{2}}^{\frac{h}{2}}W dy,$$ 
to be the average pressure across the fracture thickness.
\end{defn}

\begin{assumption}\label{const pressure}
Assume the thickness of the fracture, $h$, to be very small,
and consider the pressure to be constant across thickness. 
Assume that the flux coming into the fracture from 
the boundary $\Gamma_{f_{out}}$ is negligible. Namely, we assume
$$ \int_{\Gamma_{f_{out}}}  \mathbf{v}_f \cdot \mathbf{n}_f ds = 0,$$
or, equivalently,
$$ \partial_x W = \partial_x \bar W_f = 0 \mbox{ on }\Gamma_{f_{out}}.$$
\end{assumption}

\begin{theorem}
Under Assumption \ref{const pressure}, Eq.~\eqref{frac pressure} and Eq.~\eqref{flux_cont} can be reduced to
\begin{equation}
-\frac{Q}{\|\Omega\|}h=h\partial_x \left(f_\beta(\|\partial_x \bar W _f\|) \partial_x \bar W_f \right)+ \jump{k_p \frac{\partial W_p}{\partial y}(x,0)} \, \, \, \mbox{ on } \Gamma_f\,, \label{new frac pressure}
\end{equation}
where $\jump{f(x,0)} = f(x,0^+)-f(x,0^-)$. 
\end{theorem}

\begin{proof}
	Using Eqs.~\eqref{frac pressure} and~\eqref{DF eqn} and then integrating over the thickness of the fracture, we have
\begin{align}
\int_{-\frac{h}{2}}^{\frac{h}{2}}-\frac{Q}{\|\Omega\|} dy &= \int_{-\frac{h}{2}}^{\frac{h}{2}} - \nabla \cdot \mathbf{v}dy\,,
\end{align}
or
\begin{align}
-\frac{Q}{\|\Omega\|} h &=- \int_{-\frac{h}{2}}^{\frac{h}{2}} \partial_x v_1 dy - \int_{-\frac{h}{2}}^{\frac{h}{2}} \partial_y v_2 dy\,,\label{eq39}
\end{align}
where $\mathbf{v}=\langle v_1,v_2 \rangle$.
The second integral on the right hand side of Eq.~\eqref{eq39} can be rewritten as 
\begin{align}
\int_{-\frac{h}{2}}^{\frac{h}{2}} \partial_y v_2 dy &=v_2\Big(x,\frac{h}{2}\Big)-v_2\Big(x,-\frac{h}{2}\Big) \nonumber \\
&=  \mathbf{v}_f\Big(x,\frac{h}{2}\Big) \cdot \mathbf{n}_f + \mathbf{v}_f\Big(x,-\frac{h}{2}\Big) \cdot \mathbf{n}_f  \nonumber \\
&= - \mathbf{v}_p\Big(x,\frac{h}{2}\Big) \cdot \mathbf{n}_p - \mathbf{v}_p\Big(x,-\frac{h}{2}\Big) \cdot \mathbf{n}_p \nonumber \\
&= -\Bigg[-k_p \left(\frac{\partial W_p(x,\frac{h}{2})}{\partial x},\frac{\partial W_p(x,\frac{h}{2})}{\partial y} \right) \cdot\langle0,-1\rangle\\ 
&\qquad\quad- k_p \left(\frac{\partial W_p(x,-\frac{h}{2})}{\partial x},\frac{\partial W_p(x,-\frac{h}{2})}{\partial y} \right) \cdot\langle 0,1\rangle \Bigg] \nonumber\\ 
&=-\Bigg[k_p \left( \frac{\partial W_p(x,\frac{h}{2})}{\partial y}-\frac{\partial W_p(x,-\frac{h}{2})}{\partial y}\right) \Bigg]\,. \nonumber
\end{align}
Since $h$ is assumed to be very small 
\begin{align}
\int_{-\frac{h}{2}}^{\frac{h}{2}} \partial_y v_2 dy &=-\Bigg[k_p \left( \frac{\partial W_p(x,0^+)}{\partial y}-\frac{\partial W_p(x,0^-)}{\partial y}\right) \Bigg] \nonumber\\
&=-\jump{k_p \frac{\partial W_p}{\partial y}(x,0)}\,.
\end{align}
which follows from Eq.~\eqref{flux_cont}, the continuity of the flux on the boundary. Here $\langle 0,1\rangle$ is the upward unit normal vector and $\langle0,-1\rangle$ is the downward unit normal vector to the fracture.\\
The first integral on the right hand side of Eq.~\eqref{eq39} can be rewritten as
\begin{align}
\int_{-\frac{h}{2}}^{\frac{h}{2}}  \partial_x v_1 dy &= \int_{-\frac{h}{2}}^{\frac{h}{2}}  \left(\partial_x (- f_\beta\left(\|\partial_x \bar W _f\|\right) \partial_x \bar W_f)\right) dy \nonumber \\
&=-\partial_x\left(f_\beta(\|\partial_x \bar W _f\|) \partial_x \int_{-\frac{h}{2}}^{\frac{h}{2}}  \bar W_f dy\right) \nonumber \\
&=-\partial_x\left(f_\beta(\|\partial_x \bar W _f\|) \partial_x(h \bar W_f)\right)\,.
\end{align}
Therefore, we obtain Eq.\eqref{new frac pressure}.
\end{proof}
Finally, the system~\eqref{porous pressure} - \eqref{ini well} can be reduced,
and the model for the filtration process can be formulated 
as 
\begin{align}
-\frac{Q}{\|\Omega\|} &=  \nabla \cdot k_p \nabla W_p & \hspace{-1cm}\mbox{ in } \Omega_p = \Omega \setminus \Gamma_f \,,\label{fin porous pressure}\\
\mathbf{v}_p \cdot \mathbf{n}_p &=  0  & \mbox{ on } \Gamma_{out} \, ,\\
W_p &= \bar W_f  & \mbox{ on } \Gamma_f \,,\\
-\frac{Q}{\|\Omega\|}h &= h\partial_x\left(f_\beta(\|\partial_x \bar W _f\|) \partial_x\bar  W_f\right) +\jump{k_p \frac{\partial W_p}{\partial y}(x,0)} & \mbox{ on } \Gamma_f\,, \label{jump} \\
W_p &= \bar W_f = 0  & \mbox{ on } \Gamma_w \,,\label{well pressure}\\
\partial_x \bar W &= 0 & \mbox{ on } \Gamma_{f_{out}} \,.\label{no flux}
\end{align}

\subsection{Estimates for the difference between solutions of original and reduced fracture models.}
In this section, we justify the introduction of the reduced model \eqref{fin porous pressure}-\eqref{no flux}. We analyze the difference between the solution for the original flow equation and the solution for the reduced fracture model in appropriate norms.
We consider two types of flows in the fracture: isotropic (Section \ref{isotropic}) and anisotropic (Section \ref{anisotropic}).

We denote $\Gamma_f^\pm$ to be the top/bottom boundary (at $\lambda= \pm \frac{h}{2}$) of the fracture $\Omega_f$ and $\mathbf{n}_f^{\pm}$ to be the unit outward normal vector on $\Gamma_f^\pm$.
We investigate the stability of the reduced problem defined in the fracture domain $\Omega_f$ with given flux conditions on the boundaries $\Gamma_f^\pm, \Gamma_{f_{out}}$ and given pressure at $\Gamma_w.$ 
Let $q^\pm(x)$ be given fluxes on $\Gamma_f^\pm$. Then $\jump{k_p \frac{\partial W_p}{\partial y}(x,0)}$ in Eq.~\eqref{jump} equals to $-q^+(x)-q^-(x).$

Next we prove closeness of the solutions of the reduced problem to the original one. Let $W=W(x,y)$ be the actual pressure inside  fracture $\Omega_f$ and $\bar W=\bar{W}(x)$ be the pressure solution of the reduced equation. Let $\nabla=\left<\partial_x,\partial_y \right>$. 

\subsubsection{Isotropic Case}  \label{isotropic}
In this section, we provide estimates for the difference between two solutions for the isotropic case.
The analysis will be based on the following results.

\begin{lemma} \label{lidia}
For $f_{\beta}(\|\nabla W \|)$ defined by Eq.~\eqref{non lin term}, $1 \leq q < 2$,
\begin{align}
&\int_\Omega \Big( f_\beta \left(\|\nabla W _1\| \right) \nabla W_1-  f_\beta \left(\|\nabla W _2\| \right) \nabla  W_2 \Big) 
\cdot \nabla \left(W_1- W_2\right)   d\Omega
 \nonumber \\
 &\qquad   \geq C \left\|\nabla \left(W_1- W_2\right) \right\|_{L^q} ^2
 \Big \{1+ \max 
 \left( \|\nabla W_1\|_{L^{\frac{q}{2(2-q)}}}, \|\nabla W_2\|_{L^{\frac{q}{2(2-q)}}}\right)\Big \}^{-1/2}\,.
\end{align}
\end{lemma}
\begin{proof}
The proof follows from Lemma III.11 in Ref \cite{aulisa2009analysis} with $a=1/2$.
\end{proof}

Now, we investigate   the difference between the solutions of the two problems defined below:\\

1. The flow equation for the original fracture is given by

\begin{equation}
-\frac{Q}{\|\Omega\|} =\nabla \cdot f_\beta(\|\nabla W\|)\nabla W  \,, \label{new1L}
\end{equation}
with the boundary conditions
\begin{align}
-f_\beta (\|\nabla W\|)\nabla W \cdot \mathbf{n}_f^+&= q^+(x), &&\mbox{ on } \Gamma_f^+, \label{bc1L} \\
-f_\beta (\|\nabla W\|)\nabla W \cdot \mathbf{n}_f^-&= q^-(x), &&\mbox{ on } \Gamma_f^-, \label{bc2L} \\
W  &=0, &&\mbox{ on } \Gamma_w,\label{bc3L}\\
-f_\beta (\|\nabla W\|) \frac{\partial W }{\partial \mathbf n_f}  &=0, &&\mbox{ on } \Gamma_{f_{out}}\label{bc4L}\,.
\end{align}

2. The flow equation for the reduced 1-D fracture model, formulated as a 2-D problem, is given by 

\begin{equation}
-\nabla \cdot f_\beta(\|\nabla \bar W\|)\nabla \bar W  
=\frac{Q}{\|\Omega\|} - \frac{1}{h} \left(q^+(x)+q^-(x) \right)  \mbox{ in } \Omega_f \,, \label{new2L}
\end{equation}
with boundary conditions
\begin{align}
-f_\beta (\|\nabla \bar W\|)\nabla \bar W \cdot \mathbf{n}_f^\pm&= 0,& &\mbox{ on } \Gamma_f^+ \cup \Gamma_f^- \,,\label{bc1aL}\\
\bar W  &=0 ,& &\mbox{ on } \Gamma_w,\label{bc2aL}\\
-f_\beta (\|\nabla \bar W\|) \frac{\partial \bar W }{\partial \mathbf n_f}  &=0,& &\mbox{ on } \Gamma_{f_{out}}\label{bc3aL} \,.
\end{align}
Clearly the solution of the problem \eqref{new2L} - \eqref{bc3aL}  is $y$ independent.

\begin{lemma} \label{boundness}
There exists a constant $C$ depending on $\Omega$, $Q$, $q^+$ and $q^-$ such that the corresponding basic profiles $W$ and $\bar W$ satisfy
$$\|\nabla W\|_{L^\frac{3}{2}(\Omega)} \leq C, \qquad \|\nabla \bar W\|_{L^\frac{3}{2}(\Omega)} \leq C.$$
\end{lemma}
\begin{proof}
The proof follows from Theorem V.4 in Ref \cite{aulisa2009analysis} with $a=1/2$.
\end{proof}

\begin{theorem}
\begin{align}
\|W_x- \bar W_x \|^2_{L^\frac{3}{2}(\Omega_f)}
+ \|W_y \|^2_{L^\frac{3}{2}(\Omega_f)}
\leq C  \left( \|q^+ \|^2_{L^3(\Omega_f)}+ \|q^-\|^2_{L^3(\Omega_f)} \right)\,.
\end{align}

 \end{theorem}

 \begin{proof}
Subtracting Eq.~\eqref{new2L} from Eq.~\eqref{new1L}, multiplying by $W- \bar W$ and integrating over the volume of the fracture we obtain,
\begin{align}
&\int_0^L \int_{-\frac{h}{2}}^{\frac{h}{2}} 
-\nabla \cdot \Big( f_\beta(\|\nabla W\|)\nabla W - f_\beta(\|\nabla \bar W\|)\nabla \bar W \Big)\left(W- \bar W\right)\, dy dx \nonumber \\
&\qquad \qquad \qquad = \int_0^L \int_{-\frac{h}{2}}^{\frac{h}{2}}  \frac{1}{h} \left(q^+(x)+q^-(x) \right)(W- \bar W )\, dy dx\,.
\end{align}
Using Green's formula (See Ref \cite{evans1997partial}) and boundary conditions \ref{bc3L} and \ref{bc4L}, the above equation is equivalent to
\begin{align}
I_1=I_2+I_3\,, \label{I123L}
\end{align}
where
\begin{align}
I_1=&\int_0^L \int_{-\frac{h}{2}}^{\frac{h}{2}}\Big( f_\beta(\|\nabla W\|)\nabla W - f_\beta(\|\nabla \bar W\|)\nabla \bar W \Big)\cdot \nabla \left(W- \bar W\right)\, dy dx \,, \label{I1L} \\
I_2= & \int_0^L \int_{-\frac{h}{2}}^{\frac{h}{2}}  \frac{1}{h} \left(q^+(x)+q^-(x) \right)(W- \bar W )\, dy dx\,, \label{I2L}\\
I_3= & \int_{0}^{L}\Big( f_\beta(\|\nabla W\|)\nabla W - f_\beta(\|\nabla \bar W\|)\nabla \bar W \Big)\cdot \mathbf{n}^+_f(W- \bar W )\Big \rvert_{\frac{h}{2}}\,\, dy \nonumber \\
+ & \int_{0}^{L}\Big( f_\beta(\|\nabla W\|)\nabla W - f_\beta(\|\nabla \bar W\|)\nabla \bar W \Big)\cdot \mathbf{n}^-_f(W- \bar W )\Big \rvert_{-\frac{h}{2}}dy \,.\label{I3L}
\end{align}
Consider $I_1$.
By lemma \ref{boundness} and \ref{lidia} with $q=3/2$, there exist positive constants $C_0$ and $C_1$ such that,
 \begin{align}
I_1 &\geq C_0 
 \left\|\nabla \left(W- \bar W\right) \right\|_{L^\frac{3}{2}}^{2} \ge 
 C_1  \left( 
 \left\|W_x- \bar W_x\right\|_{L^\frac{3}{2}}^{2} + 
 \left\|W_y- \bar W_y\right\|_{L^\frac{3}{2}}^{2} \right)\nonumber\\
 &= 
 C_1  \left( 
 \left\|W_x- \bar W_x\right\|_{L^\frac{3}{2}}^{2} + 
 \left\|W_y\right\|_{L^\frac{3}{2}}^{2} \right)\,.
 \label{I_1new}
\end{align}
Now consider $I_3$.
Due to the boundary conditions \eqref{bc1L}, \eqref{bc2L} and since $\bar W$ is $y$ independent, we have
\begin{align}
I_3= \int_0^L  -q^+(x)(W- \bar{W} ) \big \rvert_{\frac{h}{2}} 
-q^-(x)(W- \bar{W}) \big \rvert_{-\frac{h}{2}} \,\, dx.
\end{align}
It then follows that
\begin{align}
I_2+I_3&=\int_0^L q^+(x) \int_{-\frac{h}{2}}^{\frac{h}{2}}\frac{ W(x,y) - W\left(x,\frac{h}{2}\right)}{h}\,\,dy\,dx \nonumber \\ 
&+\int_0^L q^-(x) \int_{-\frac{h}{2}}^{\frac{h}{2}}\frac{ W(x,y) - W\left(x,-\frac{h}{2}\right)}{h}\,\,dy\,dx \,. \label{I2+3L}
\end{align}
 Using Holder, Cauchy and Poincare inequalities is not difficult to show that
\begin{align}
 | I_2+I_3| &\le \frac{1}{4 \varepsilon}
 \Bigg[ \left( \int_0^L \int_{-\frac{h}{2}}^{\frac{h}{2}} \left|q^+\right|^3   dy \, dx \right)^\frac{2}{3}+
 \left( \int_0^L \int_{-\frac{h}{2}}^{\frac{h}{2}}  \left|q^-\right|^3  dy \, dx \right)^\frac{2}{3} \Bigg]\nonumber \\
&\qquad +2\varepsilon \left(\frac{2}{3}\right)^\frac{4}{3} \left(\int_0^L \int_{-\frac{h}{2}}^{\frac{h}{2}} \left| W_y \right|^\frac{3}{2} \,dy\, dx\right)^\frac{4}{3}\,. \label{modI2+3L}
\end{align}
Combining Eqs.~\eqref{I_1new}, \eqref{I2+3L} and \eqref{modI2+3L}, choosing
$\varepsilon=\frac{C_1}{4}\left(\frac{3}{2}\right)^\frac{4}{3}$ and setting $C_2=\frac{1}{C_1} \left(\frac{2}{3}\right)^\frac{4}{3} $
yields
\begin{align}
& C_1 \left\|W_x- \bar W_x\right\|_{L^\frac{3}{2}}^{2} + 
 \frac{C_1}{2} \left\|W_y\right\|_{L^\frac{3}{2}}^{2} \le
C_2 \left( \|q^+ \|^2_{L^3(\Omega_f)} +\|q^-\|^2_{L^3(\Omega_f)} \right).
%
%
%
\end{align}
Therefore, we have
\begin{align}
\|W_x- \bar W_x \|^2_{L^\frac{3}{2}(\Omega_f)}
+ \|W_y \|^2_{L^\frac{3}{2}(\Omega_f)}
\leq C \left(\|q^+ \|^2_{L^3(\Omega_f)} +\|q^-\|^2_{L^3(\Omega_f)}\right)\,,
\end{align}
with $C=2 C_2/C_1$.
\end{proof}

\begin{remark}
From the theorem above, it follows that for a given fracture with thickness $h$, 
the difference between the solutions of the two problems can be controlled by the boundary data.
For $h$ going to zero, the constant $C$ in \ref{boundness} will grow since each individual solution will diverge
with order $h^{-s}$, for some $s>0$. Then, the above theorem will not be sufficient.

However, it should be noted that in the reservoir-fracture system as $h$ goes to zero, the fracture
vanishes, and the oil flows toward the well similarly to the unfractured case. 
Then as $h$ becomes smaller, $q^+$ and $q^-$ gets smaller as well, and therefore the individual velocities remain bounded. In this setting the above theorem is still valid.
\end{remark} 
In the next section, we will prove a better result as $h$ goes to zero,  
which will work also for unbounded individual solutions, 
by making a stronger assumption on the type of the non-linearity the flow is subjected to.

 \subsubsection{Anisotropic Case}\label{anisotropic}
In this section, we provide estimates for the difference between two solutions for small fracture thickness 
assuming the flow to be anisotropic. 

In particular we assume that only the velocity component parallel to the fracture is subjected to the Forchheimer equation, 
while the component perpendicular to the fracture is subjected to Darcy law. Namely,
we assume that in \eqref{F eqn}, $\mathbf{v} = (v_1,v_2)$ is subjected to the non-linear anisotropic equation of the form
\begin{align*} 
&- p_x-\frac{\mu}{k}v_1-\beta|v_1|v_1=0\,,\\
&- p_y-\frac{\mu}{k}v_2=0\,.
\end{align*}
   
In this case the non-linear function in \eqref{non lin term} becomes the tensor 
$$F_\beta(\bm{\eta})= \begin{bmatrix}
    f_\beta(|\eta_1|)      & 0 \\
    0   &   k
   \end{bmatrix}\,, $$ where $\bm \eta =(\eta_1,\eta_2)$, and $f_{\beta}(|\eta_1|)=\frac{2}{\alpha+\sqrt {\alpha^2 +4 \beta |\eta_1|}}.$
   
In the analysis, we will use the following properties.
\begin{lemma}\label{NSM} 
For $f_{\beta}(|\eta |)$ defined as above, the function $f_{\beta}(|\eta |)|\eta |$ is strictly monotonic on bounded sets. More precisely,
\begin{align}
\Big(f_{\beta}(|\eta_1 |) \eta_1 -& f_{\beta}( |\eta_2 |) \eta_2 \Big)(\eta_1 - \eta_2 )  \nonumber \\
& \geq \frac{1}{2}f_{\beta}\left(\max (|{\eta}_1| ,|{\eta}_2| )\right)({\eta}_1 - {\eta}_2 ) ^2\,.
\end{align}
\end{lemma}
\begin{proof}
The proof follows from Proposition III.6 in Ref. \cite{aulisa2009analysis}), where $\lambda=1$ and $\bm \eta$ is one dimensional. 
\end{proof}

\begin{lemma} \label{g(u)}
Let $$ g( u ) = \left( \sqrt{1+|u|}-1\right)\sign(u).$$ Then
$$ g'( u ) = \frac{1}{2 \sqrt{1+|u|}}\le 1/2,$$ and
$$ (g(u)-g(v))^2 \ge (\sqrt{0.5 u}~\sign(u) - \sqrt{0.5 v}~\sign(v))^2 \quad \mbox{ if } \max(|u|,|v|)\ge 24.$$
\end{lemma}
\begin{proof}
Let $u\ge 24$ fixed. It is not difficult to show that the piecewise defined function
\begin{align*}
l(v) = \left\{
\begin{array}{l l}
 l^-(v) = (g(u)-g(v))^2 - (\sqrt{0.5u} + \sqrt{ - 0.5v})^2 & \mbox { for } v<0 \\
 l^+(v) = (g(u)-g(v))^2 - (\sqrt{0.5u} - \sqrt{   0.5v})^2  & \mbox { for } v\ge0
\end{array}
\right.
\end{align*}
is continuous, differentiable and satisfies $l(v)\ge 0$, for all $v$. 
In particular, for any given $u\ge 24$, the function $l(v)$ has three critical points $v_1$, $v_2$ and $v_3$,
such that $v_1<0<v_2<v_3=u$, 
with $v_3$ being an absolute minimum and $l(v_3)=0$.

For $u\le -24$ the proof is similar. Moreover by symmetry of $u$ and $v$, the same argument can be repeated by fixing $|v|\ge24$ and letting $u$
vary.
\end{proof} 

Now, we investigate   the difference between the solutions of the two problems defined below:

1. The flow equation for the original fracture is given by

\begin{align}
-\frac{\partial}{\partial x} \left( f_\beta \left(\left|W_x \right|\right)W_x  \right)-\frac{\partial}{\partial y}\left( k W_y \right) =\frac{Q}{\|\Omega\|}     \mbox{ in } \Omega_p  \label{new1}  \,,
\end{align}
with boundary conditions
\begin{align}
-k \nabla W \cdot \mathbf{n}_f^+& =q^+(x),& &\mbox{ on } \Gamma_f^+ \label{bc1+} \,,\\
-k \nabla W \cdot \mathbf{n}_f^-& =q^-(x),& &\mbox{ on } \Gamma_f^-  \,, \label{bc1-} \\
W&=0,&  &\mbox{ on } \Gamma_w  \,,\label{bc2} \\
-f_\beta (|W_x|) \frac{\partial W }{\partial \mathbf n_f}  &=0,
&&\mbox{ on } \Gamma_{f_{out}}  \,.\label{bc3}
\end{align}

2. The flow equation for the reduced fracture model is given by
\begin{equation}
-\frac{\partial}{\partial x} \left( f_\beta \left(\left|\bar W_x \right|\right)\bar W_x \right) =\frac{Q}{\|\Omega\|} - \frac{1}{h} \left(q^+(x)+q^-(x) \right)  \mbox{ in } \Omega_f \,, \label{new2}
\end{equation}
with boundary conditions
\begin{align}
-k \nabla \bar W \cdot \mathbf{n}_f^\pm& =0,& &\mbox{ on } \Gamma_f^+ \cup \Gamma_f^- \,,\label{bc1a}\\
 \bar W &= 0 ,& &\mbox{ on } \Gamma_w   \,, \label{bc2a} \\
-f_\beta (|\bar W_x|) \frac{\partial \bar W }{\partial \mathbf n_f}  &= 0,&  &\mbox{ on } \Gamma_{f_{out}} \,. \label{bc3a} 
\end{align}

\begin{theorem}
\begin{align}
\int_0^L \int_{-\frac{h}{2}}^{\frac{h}{2}}\Big( f_\beta \left(\left|W_x \right|\right)W_x  
-  & f_\beta\left( \left|\bar W_x \right| \right)\bar W_x \Big)
(W_x- \bar W_x ) 
+ \frac{k}{2} W_y^2 \,\, dy\, dx  \nonumber  \\
&\le \frac{h}{2 k} \int_0^L  \left(q^-\right)^2 + \left(q^+\right)^2    dx.
\end{align}
Moreover,
\begin{align}
\int_0^L \int_{-\frac{h}{2}}^{\frac{h}{2}}&  \frac{k}{2} \frac{\alpha^2}{\beta} \left( \sqrt{\left| W_x \right|}\sign 
\left( W_x\right) - \sqrt{\left|\bar W_x \right|}\sign \left(\bar W_x\right)  \right)^2 
\Hs\left(W_x, \bar W_x \right) \nonumber \\
& +\frac{k}{6}\left(
 W_x- \bar W_x\right)^2 \left( 1 -\Hs\left(W_x, \bar W_x \right) \right) 
+ \frac{k}{2}\,  W_y^2 \,\, dy\, dx \nonumber \\
 & \qquad  \leq
 \frac{h}{2 k} \int_0^L  \left(q^+\right)^2 + \left(q^-\right)^2 dx \,,
\end{align}
where
\begin{equation}\label{Hsdef}
     \Hs(\zeta,\eta)=
    \begin{cases}
      1, & \text{if } \max(\left|\zeta\right|, \left| \eta\right|) \ge \frac{6 \alpha^2}{\beta}  \\
      0, & \text{otherwise}
    \end{cases}.
   \end{equation}
\end{theorem}
\begin{proof}
Subtracting Eq.~\eqref{new2} from Eq.~\eqref{new1}, multiplying by $W- \bar W$ and integrating over the volume of the fracture we obtain,
\begin{align}
\int_0^L \int_{-\frac{h}{2}}^{\frac{h}{2}} \Bigg(-\frac{\partial}{\partial x}& \Big( f_\beta \left(\left| W_x \right|\right) W_x
 -f_\beta \left(\left|\bar W_x\right|\right)\bar W_x \Big)-\frac{\partial}{\partial y} k  W_y\Bigg)\left(W- \bar W\right)\, dy dx \nonumber \\
&= \int_0^L \int_{-\frac{h}{2}}^{\frac{h}{2}}  \frac{1}{h} \left(q^+(x)+q^-(x) \right)(W- \bar W )\, dy dx\,.
\end{align}
Using Green's formula (See Ref \cite{evans1997partial}) and boundary conditions \ref{bc2} and \ref{bc3}, the above equation is equivalent to
\begin{align}
I_1=I_2+I_3\,, \label{I123}
\end{align}
where
\begin{align}
I_1=&\int_0^L \int_{-\frac{h}{2}}^{\frac{h}{2}}\Big(f_\beta \left(\left|W_x \right|\right)W_x  
-   f_\beta\left( \left|\bar W_x \right| \right)\bar W_x \Big)
\left(W_x- \bar W_x \right) \nonumber \\
&\qquad \qquad +k \, W_y ^2 \,\, dy\, dx \,, \label{I1} \\
I_2= & \int_0^L \int_{-\frac{h}{2}}^{\frac{h}{2}}  \frac{1}{h} \left(q^+(x)+q^-(x) \right)(W- \bar W )\, dy dx\,, \label{I2}\\
I_3= & \int_{0}^{L}\Big( f_\beta \left(\left| W_x \right|\right) W_x  
-f_\beta\left( \left|\bar W_x \right| \right)\bar W_x \Big)\cdot \mathbf{n}^+_f(W- \bar W )\Big \rvert_{\frac{h}{2}}\,\, dy \nonumber \\
& +\int_{0}^{L}\Big( f_\beta \left(\left| W_x \right|\right) W_x  
-f_\beta\left(  \bar W_x  \right) \bar W_x \Big)\cdot \mathbf{n}^-_f(W- \bar W )\Big \rvert_{-\frac{h}{2}}dy \,.\label{I3}
\end{align}
Consider $I_3$.
Due to the boundary conditions \eqref{bc1+}, \eqref{bc1-} and since $\bar W$ is $y$ independent, we have
\begin{align}
I_3= \int_0^L  -q^+(x)(W- \bar{W} ) \big \rvert_{\frac{h}{2}} 
-q^-(x)(W- \bar{W}) \big \rvert_{-\frac{h}{2}} \,\, dx.
\end{align}
It then follows that
\begin{align}
I_2+I_3&=\int_0^L q^+(x) \int_{-\frac{h}{2}}^{\frac{h}{2}}\frac{ W(x,y) - W\left(x,\frac{h}{2}\right)}{h}\,\,dy\,dx \nonumber \\ 
&+\int_0^L q^-(x) \int_{-\frac{h}{2}}^{\frac{h}{2}}\frac{ W(x,y) - W\left(x,-\frac{h}{2}\right)}{h}\,\,dy\,dx \,. \label{I3+4}
\end{align}
Using Cauchy and Poincare inequalities we have
\begin{align}
 | I_2+I_3| \le \frac{h}{4 \varepsilon} \int_0^L  \left(q^+\right)^2 + \left(q^-\right)^2    dx +
\varepsilon \int_0^L \int_{-\frac{h}{2}}^{\frac{h}{2}} W_y^2 \,dy\, dx\,. \label{modI2+3}
\end{align}
Combining Eqs.~\eqref{I123}, \eqref{I1} and \eqref{modI2+3} and selecting $\varepsilon=k/2$, we get
\begin{align}
\int_0^L \int_{-\frac{h}{2}}^{\frac{h}{2}}\Big( f_\beta \left(\left|W_x \right|\right)W_x  
-  & f_\beta\left( \left|\bar W_x \right| \right)\bar W_x \Big)
(W_x- \bar W_x ) 
+ \frac{k}{2} W_y^2 \,\, dy\, dx  \nonumber  \\
&\le \frac{h}{2 k} \int_0^L  \left(q^-\right)^2 + \left(q^+\right)^2    dx. \label{ineq}
\end{align}
Now, using lemma \ref{g(u)} we obtain
$$ f_\beta \left(\left|W_x \right|\right)W_x = k \frac{\alpha^2}{2\beta}\left( \sqrt{1+\left| 4 \frac{\beta}{\alpha^2} W_x \right|}-1\right)\sign \left ( 4 \frac{\beta}{\alpha^2} W_x \right)=k \frac{\alpha^2}{2\beta} \; g\left ( 4 \frac{\beta}{\alpha^2} W_x \right),$$
and
$$ \frac{\partial g \left( 4 \frac{\beta}{\alpha^2} W_x \right) }{\partial W_x} \le 2 \frac{\beta}{\alpha^2}\,. $$
It then follows that
$$ |W_x -\bar W_x| \ge \frac{\alpha^2}{2 \beta} \left| g \left ( 4 \frac{\beta}{\alpha^2} W_x \right) -   g\left( 4 \frac{\beta}{\alpha^2} \bar W_x \right)\right|\,,$$
and
\begin{align}
\Big(f_\beta \left(\left|W_x \right|\right)W_x  
- &  f_\beta\left( \left|\bar W_x \right| \right)\bar W_x \Big)
(W_x- \bar W_x ) \nonumber \\
 &\ge 
k \left(\frac{\alpha^2}{2 \beta}\right)^2 \left[g \left ( 4 \frac{\beta}{\alpha^2} W_x \right) -   g\left( 4 \frac{\beta}{\alpha^2} \bar W_x \right) \right]^2 \label{dif}\,.
\end{align}
If $ \max(|W_x|,|\bar W_x|)\ge 6 \frac{\alpha^2}{\beta}$, due to Lemma \ref{g(u)}, the right hand side of the above inequality is greater than or equal to
\begin{align}
& 
k \left(\frac{\alpha^2}{2 \beta}\right)^2 
\left( \sqrt{2 \frac{\beta}{\alpha^2} \left| W_x \right|}
\sign \left( 4 \frac{\beta}{\alpha^2} W_x\right) - 
\sqrt{2 \frac{\beta}{\alpha^2} \left|\bar W_x \right|}
\sign \left(4 \frac{\beta}{\alpha^2}\bar W_x\right)\right)^2\nonumber\\
&=
\frac{k}{2} \frac{\alpha^2}{\beta} 
\left( 
\sqrt{ \left| W_x \right|} \sign \left(  W_x\right) - 
\sqrt{ \left|\bar W_x \right|} \sign \left(\bar W_x\right)
\right)^2 \,. \label{greater}
\end{align}
On the other hand, if $ \max(|W_x|,|\bar W_x|)< 6 \dfrac{\alpha^2}{\beta}$, due to Lemma \ref{NSM}, we have
\begin{align}
\Big(f_\beta \left(\left|W_x \right|\right)W_x  
-&   f_\beta\left( \left|\bar W_x \right| \right)\bar W_x \Big)
(W_x- \bar W_x ) \nonumber \\
& \geq \frac{1}{2} f_\beta \left(6 \frac{\alpha^2}{\beta}\right) \left(
 W_x- \bar W_x\right)^2 
  = \frac{k}{6}  \left(
 W_x- \bar W_x\right)^2\,. \label{less}
\end{align}
Combining Eqs.~\eqref{dif}, \eqref{greater} and \eqref{less} we obtain
\begin{align}
&\int_0^L \int_{-\frac{h}{2}}^{\frac{h}{2}}\left[f_\beta \left(\left|W_x \right|\right)W_x  
-   f_\beta\left( \left|\bar W_x \right| \right)\bar W_x \right]
(W_x- \bar W_x ) \, dy\, dx  \nonumber \\
&\geq \frac{k}{2}  \int_0^L \int_{-\frac{h}{2}}^{\frac{h}{2}}  \frac{\alpha^2}{\beta}\left( \sqrt{\left| W_x \right|}\sign 
\left( W_x\right) - \sqrt{\left|\bar W_x \right|}\sign \left(\bar W_x\right)  \right)^2 
\Hs\left(W_x, \bar W_x \right)\,\, dy\, dx \nonumber \\
&\quad \quad+\frac{k}{6}\int_0^L \int_{-\frac{h}{2}}^{\frac{h}{2}} \left(
 W_x- \bar W_x\right)^2 \left( 1 -\Hs\left(W_x, \bar W_x \right) \right)\,\, dy\, dx\,. \label{Hs}
\end{align}
where $\Hs$ was defined in \eqref{Hsdef}.
Finally taking into account Eqs.~\eqref{ineq} and \eqref{Hs}, it follows that
\begin{align}
&\int_0^L \int_{-\frac{h}{2}}^{\frac{h}{2}}  \frac{k}{2} \frac{\alpha^2}{\beta}  \left( \sqrt{\left| W_x \right|}\sign 
\left( W_x\right) - \sqrt{\left|\bar W_x \right|}\sign \left(\bar W_x\right)  \right)^2 
\Hs\left(W_x, \bar W_x \right) \nonumber \\
& \qquad \quad  +\frac{k}{6}\left(
 W_x- \bar W_x\right)^2 \left( 1 -\Hs\left(W_x, \bar W_x \right) \right) 
+ \frac{k}{2}\,  W_y^2 \,\, dy\, dx \nonumber \\
 & \qquad \qquad \leq
 \frac{h}{2 k} \int_0^L  \left(q^+\right)^2 + \left(q^-\right)^2  dx \,.
\end{align}
\end{proof}
\begin{remark}
From the theorem above, it follows that in the considered anisotropic case, the estimate for the difference 
between the solutions of the original problem and the reduced problem does not depend on the size of the fracture thickness $h$. 
Taking into account that each individual solution will diverge as $h$ goes to zero, with order $h^{-s}$, for some $s>0$,
this represents a much stronger estimate with respect to the one obtained in the isotropic case.  
\end{remark} 

%
%
%
%
%
%

 \section{Optimization of the Fracture Length}
In this section, we investigate the optimal fracture length that gives the maximum diffusive capacity/PI for a reservoir domain with a horizontal line fracture, using the reduced fracture model obtained in the previous section. 
In particular, for a given pressure drawdown, PDD, we aim to find 
the optimal length of the fracture for which the diffusive capacity is maximized.

In natural reservoirs, the pressure drawdown is a physical characteristic of the oil recovery.
It is observed on the field that for constant production rate $Q$, after a short transition, the PDD
remains constant for a long time. This is the case idealized in the pseudo-steady state regime previously introduced.
Namely, PDD equals 
\begin{align}
PDD & = \bar p_\Omega(t)- \bar p_w(t) \nonumber \\
&=\frac{1}{\|\Omega\|}\int_\Omega p\, d\Omega-\frac{1}{\|\Gamma_w\|} \int_{\Gamma_w} p\, dx. \nonumber 
\end{align}
and using Eq.~\eqref{soln} yields
\begin{align}
PDD & =\frac{1}{\|\Omega\|}\int_\Omega \left(W - \gamma A + K\right) d\Omega
-\frac{1}{\|\Gamma_w\|} \int_{\Gamma_w} \left(W - \gamma A + K\right) dx \nonumber.
\end{align}
Using the boundary condition on the well, 
Eq.~\eqref{well pressure}, it follows that
\begin{align}
PDD &= \frac{1}{\|\Omega\|} \int_{\Omega} W d\Omega = const\nonumber.
\end{align}
Then  using Eq.~\eqref{DC}, the diffusive capacity (productivity index) becomes 
\begin{equation}
J_p=\frac{Q}{PDD} = \frac{Q}{\frac{1}{\|\Omega\|} \int_{\Omega} W d\Omega}\,. \label{diff_cap}
\end{equation}

For a given fracture length $L$ (for a given geometry), 
the value of $Q$ that provides a prescribed value of the 
pressure drawdown is unknown. 
We need to solve the inverse problem:
find $Q$ such that the solution of the system~\eqref{fin porous pressure}-\eqref{no flux}
gives a prescribed value of PDD.

To solve this, the following algorithm  described in \cite{aulisa2014numerical,aulisa2015practical} can be used.

\subsection{Set point control algorithm} \label{setpoint}
The following algorithm is very general and provides a recipe to easily solve
inverse problems as the one described above. All the details can be found in \cite{aulisa2014numerical,aulisa2015practical}.
Here we just sketch a constructive proof of the algorithm that we adapt for our purpose.
Consider the following {\it single input single output} (SISO) abstract system
\begin{align}
Az+Fz+B_{in}\gamma=0\,, \label{control eqn}\\
y=Cz\,,\label{Y}
\end{align}
where $z$ is the state variable, 
$A$ is some linear homogeneous operator, 
$F$ is some non linear function, $B_{in}$ is the input operator, 
and $C$ is the output operator.

The objective is to find $\gamma$ for Eq.~\eqref{control eqn} 
such that 
\begin{equation}
Cz =y_r, 
\end{equation}
for a given target $y_r$.
We assume that such a $\gamma$ exists.

\begin{lemma}
The solution of the system
\begin{equation}
\left\{
\begin{array}{r l}
AX+B_{in} &=0  \label{setpointsystem}\\ 
Az+Fz+B_{in}\gamma&=0  \\
A\tilde{z}+Fz&=0
\end{array}
\right. , 
\end{equation}
with 
\begin{equation}
\gamma = \dfrac{y_r-C\tilde{z}}{CX} \label{gamma}, 
\end{equation}
satisfies the constraint $$y_r=Cz\,.\label{target}$$
\end{lemma}

\begin{proof}
Using the first equation in the system~\eqref{setpointsystem}, we have 
\begin{equation}
X=-A^{-1}B_{in}\,.\label{lin soln}
\end{equation}
Let
\begin{align}
G&=CX = -C A^{-1}B_{in}\,.\label{G}
\end{align}
From the second and third equations in the system~\eqref{setpointsystem},
we have 
\begin{align}
z=-A^{-1}(F z+B_{in}\gamma)\, \label{sln}\,,\\
\tilde{z} =-A^{-1}(Fz) \,. \label{Z tilde}
\end{align} 
From Eqs.~\eqref{gamma}, \eqref{G} and \eqref{Z tilde}, we have 
\begin{equation}
y_r=C\tilde{z}+G\gamma =  -CA^{-1}(Fz+B_{in}\gamma). \label{y_r}
\end{equation} 
Finally, substituting Eq.~\eqref{sln} in the above equation we get
the desired identity
\begin{equation}
y_r=Cz\,.
\end{equation} 
\end{proof}

\subsection{Solution Strategy}\label{Sol_Str}

\begin{enumerate}
\item \label{item 1}
Consider the reservoir-well domain with no fracture ($L=0$). Let $W^0_p$ be the solution of the PDE system~\eqref{fin porous pressure} - \eqref{no flux} when $L=0$. Then the system reduces to
\begin{align}
\nabla\cdot\left(-k_p\nabla W^0_p\right) &= \frac{Q}{\|\Omega\|} &\mbox{ on } \Omega_p\,,\\
-k_p\nabla W^0_p\cdot \mathbf{n_p} &=0 &\mbox{ on } \Gamma_{out}\,,\\
W^0_p &=0 & \mbox{ on } \Gamma_w\,.
\end{align}
Set
$$PDD^* = \frac{1}{\|\Omega\|} \int_{\Omega_p} W^0_p d\Omega .$$
\item \label{item 2}
Consider the reservoir-well system with fracture length equal to $L$. 
Let $W(L,Q)$ be the solution of system~\eqref{fin porous pressure} - \eqref{no flux}
for given $Q$. 
The inverse problem consists in finding $Q$ such that
$PDD(L,Q) = PDD^*$.
 
Let $k_p$ be the permeability of porous medium, 
$k_f$ be the linear part of the permeability of the fracture, 
and $f_\beta$ be defined as in Eq.$\eqref{non lin term}$.
In order to use the set point algorithm described in section \ref{setpoint},
we need to introduce certain operators. 

The homogeneous linear operator $A$, applied to
the state variable $\Psi = \left\{ \Psi_p, \Psi_f \right\}$, corresponds to
\begin{equation}
A(\Psi)=
\left\{
\nabla\cdot(-k_p\nabla \Psi_p) ,\;
\jump{- k_p \frac{\partial \Psi_p}{\partial y}(x,0)}   +\, h\nabla\cdot(-k_f\nabla \Psi_f) \\ 
\right\}, 
\end{equation}
with domain
\begin{align}
D&(A)= \big\{  \Psi_p \in H^1(\Omega_p), \, \Psi_f \in H^1(\Gamma_f) : 
 \; \left. \Psi_p \right|_{\Gamma_f} = \Psi_f, \; \nonumber \\
&  \left. -k_p \nabla \Psi_p \cdot \mathbf{n}_p \right|_{\Gamma_{out}}  =0, 
\left. \Psi_f\right|_{\Gamma_w} =0,\; 
\left. -k_f \partial_x \Psi_f \right|_{\Gamma_{f_{out}}} = 0 \big\}.
\end{align}

The non linear operator $F$ corresponds to
\begin{align}
F(\Psi) =  \left\{ 0 , \; 
-\, h\nabla\cdot(-k_f\nabla \Psi_f) + \, h\nabla\cdot(-f_{\beta}\nabla \Psi_f)
\right\}.
\end{align}

The input operator $B_{in}$ corresponds to
\begin{equation}
B_{in} =
\left\{
 - \frac{1}{\|\Omega\|} ,\,
 -  \frac{h}{\|\Omega\|}
\right\}\,.
\end{equation}

The output operator $C$ corresponds to
\begin{align}
C(\Psi) & = \frac{1}{\|\Omega\|} \int_\Omega \Psi\,d\Omega \nonumber \\
& = \frac{1}{ \|\Omega_p\| + h\,L } \left(\int_{\Omega_p} \Psi_p\, d\Omega  + h \int_{\Gamma_f} \Psi_f dx 
\right).
\end{align}

These operators will enable us to solve the system \eqref{setpointsystem}, where we set $y_r = PDD^*$.

\end{enumerate}

The diffusive capacity $J_p$ will be determined using Eq.~\eqref{diff_cap}, for a given fracture length $L$ and for a given value of $\beta$. We investigate this problem for different reservoir geometries, different fracture lengths and different values of $\beta$.

\subsection{Numerical Simulations}

In this section, we evaluate the diffusive capacity for different reservoir geometries with a horizontal planar fracture. We analyze how the diffusive capacity behaves as the fracture length $L$ and the parameter $\beta$ change. All the simulations have been performed using the COMSOL Multiphysics software. 

\subsubsection{Cylindrical reservoir with horizontal planar fracture} \label{cyl res}
 The following is an analysis of the diffusive capacity for a horizontal planar fracture in a cylindrical reservoir. 
\begin{figure}[!h]
\begin{center}
	\includegraphics[scale=0.35]{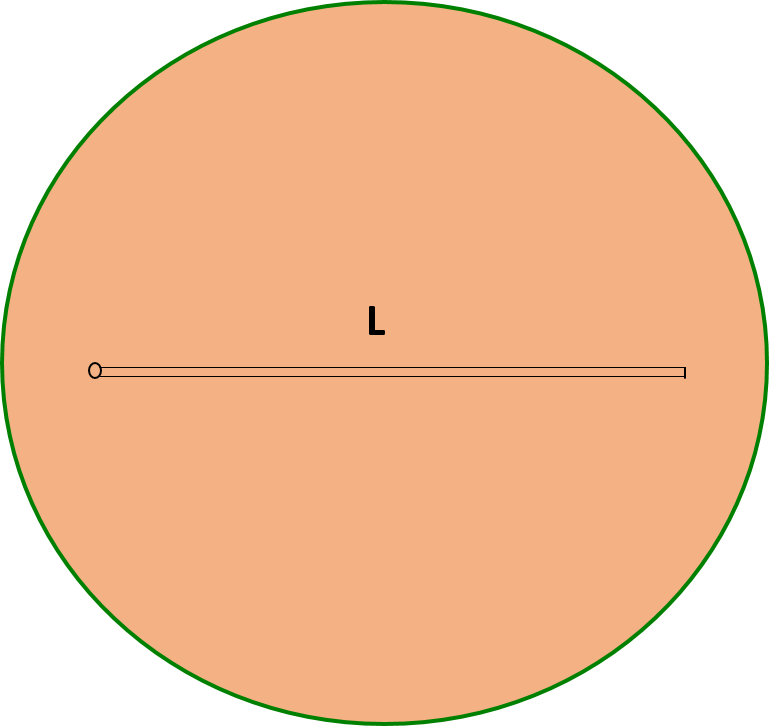}
	\caption{Horizontal planar fracture in a cylindrical reservoir}	
\end{center}
\end{figure}

We use the solution strategy described in the previous section to solve the problem of finding the production rate, which results 
in a desired pressure drawdown. Next, we calculate the corresponding diffusive capacity. 
We analyze how the diffusive capacity varies, while the fracture length and the parameter $\beta$ change.

First, we consider the reservoir-well domain with no fracture and set Q=1000. 
Solving the system in step \ref{item 1} of Section \ref{Sol_Str}, 
we obtain the value of the pressure drawdown, as $PDD^*=988.06$. The diffusive capacity for this value is $$J^*_p=\frac{Q}{PDD^*}=1.0121.$$ 
Then, we consider the reservoir-well system with fracture, where both $L$ and $\beta$ change.
Following the procedure in step \ref{item 2} of Section \ref{Sol_Str} and imposing the pressure drawdown PDD($L,\,\beta$)=PDD*, 
we evaluate the diffusive capacity $J_p(L,\,\beta)$ using Eq.~\eqref{diff_cap}.

\begin{table}[h]
	\begin{center}
		\caption{Diffusive capacity for a cylindrical reservoir.}
		\begin{tabular}{|c|c|c|c|c|c|}
		\hline
		\backslashbox[2cm]{\hspace{0.5mm}$\beta$}{\vspace{-4mm}L} & 10 & 20 & 30 & 40 & 50  \\
		\hline
		1.0E-05 &2.2345  &2.6668  &2.9045  &3.0218  &3.074\\
		1.0E-04 &1.8451  &1.9418  &1.9868  &1.9972  &1.9957\\
		1.0E-03 &1.4121  &1.4163  &1.4287  &1.4235  &1.4158\\
		1.0E-02 &1.1911  &1.1834  &1.1964  &1.1898  &1.182\\
		1.0E-01 &1.1524  &1.1451  &1.1627  &1.1566  &1.1483\\		
		\hline		
	        \end{tabular}
	\end{center}
\end{table}

It can be observed that, as the length of the fracture increases, the diffusive capacity (productivity index) increases. However, as the non linear coefficient term $\beta$ increases, the diffusive capacity decreases. For large values of $\beta$ for which the 
nonlinear term dominates, the diffusive capacity saturates - even if the length $L$ of the fracture is increased.

\subsubsection{Rectangular reservoir with horizontal planar fracture} \label{rec res}
The following is an analysis of the diffusive capacity for a horizontal planar fracture in a rectangular reservoir. 
\begin{figure}[h]
\begin{center}
	\includegraphics[scale=0.3]{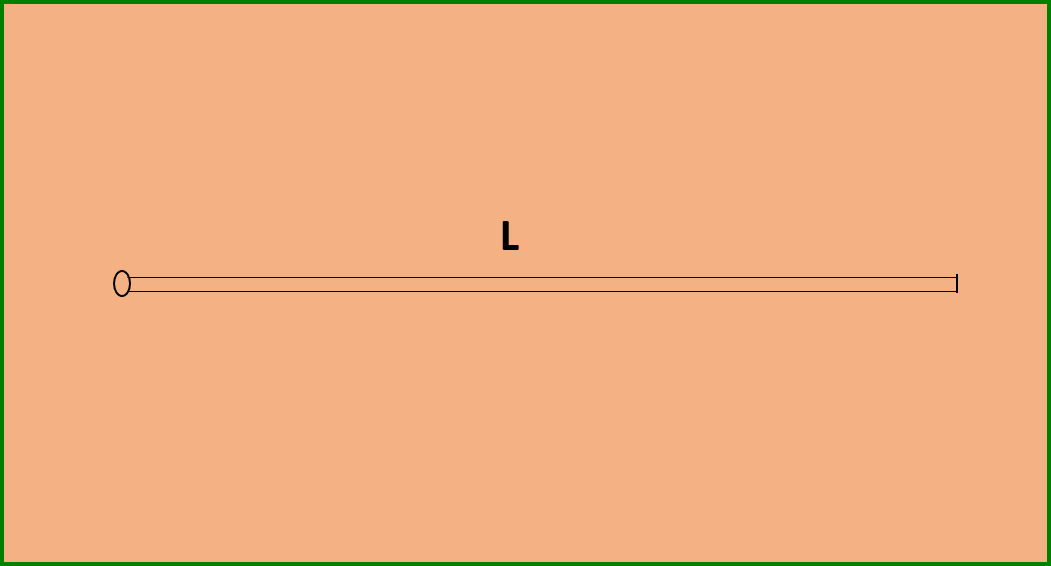}
	\caption{Horizontal planar fracture in a rectangular reservoir}	
\end{center}
\end{figure}

We set Q=1000 and carry out the same procedure as in \ref{cyl res}. For the reservoir with no fracture, we obtain the pressure drawdown equal to $PDD^*=1023.9$.
The diffusive capacity for this value is $J^*_p=0.97664$.
Then we consider the reservoir with fracture and calculate the diffusive capacity $J_p(L.\beta)$ while the length $L$ of the fracture and the parameter $\beta$ change.

\begin{table}[h]
	\begin{center}
		\caption{Productivity Index for a rectangular reservoir.}
		\begin{tabular}{|c|c|c|c|c|c|}
		\hline
		\backslashbox[2cm]{\hspace{0.5mm}$\beta$}{\vspace{-4mm}L} & 10 & 20 & 30 & 40 & 50  \\
		\hline
		1.0E-05 &2.1848  &2.7026  &2.9917  &3.1925  &3.2631 \\
		1.0E-04 &1.8265  &1.9938  &2.0333  &2.0829  &2.0718 \\
		1.0E-03 &1.4154  &1.5039  &1.4806  &1.5077  &1.4838 \\
		1.0E-02 &1.2155  &1.3266  &1.2724  &1.3022  &1.2735 \\
		1.0E-01 &1.2635  &1.3829  &1.2812  &1.3196  &1.284  \\		
		\hline		
	        \end{tabular}
	\end{center}
\end{table}

We can see that the diffusive capacity increases while increasing the length of the fracture, but decreases when the nonlinear coefficient term $\beta$ increases. Also, it can be observed that, for large values of $\beta$ for which the nonlinear term dominates, the diffusive capacity saturates even if the length of the fracture increases.

\section{Conclusion}
In this paper, we investigated Forchheimer flows in bounded reservoir domains with fractures. Generalizing our previous work and characterizing the constraints on the flow, we formulated an IBVP that describes the oil filtration process of fractured reservoir domain. We developed a mathematical approach to identify the productivity index that measures the efficiency of the well. We introduced the notion of {\it diffusive capacity}, as an integral functional over a solution of the IBVP. We established a steady state auxiliary problem which gives a time invariant solution for the IBVP and obtained the time invariant diffusive capacity.

We analyzed the difference between the solution of the original problem with 2-D flow inside the fracture and the one-dimensional one in appropriate norms, for two types of flows, isotropic and anisotropic. For the isotropic case we proved the closeness of the solutions of the two problems. In particular, we showed that, for a given fracture domain, the difference between the solutions of the two problems can be controlled by the boundary data. 

For anisotropic flow, where non-linear Forchheimer equation holds only along the fracture, while the flow remains linear in the direction perpendicular to the fracture, we  proved a stronger result. Namely,  even if the individual solutions are unbounded as the fracture thickness approaches zero, the difference between the solutions of the two problems is bounded, and does not depend on the size of the fracture thickness. Therefore we validate the introduction of the reduced model and conclude that the reduced one dimensional model of the 
fracture provides an accurate solution to approximate the actual 2-D flow, thus can be used in whole domain reservoir simulators.

We developed an approach to investigate the optimal fracture length that maximizes the diffusive capacity using a set point control algorithm. We considered different reservoir geometries, and analyzed how the diffusive capacity changes  while the fracture length and the Forchheimer  $\beta$ factor  vary. Numerical simulations confirmed that the diffusive capacity increases with increasing fracture length, but decreases with increasing $\beta$ values. Moreover, we observed that, for large values of $\beta$ for which the nonlinear term dominates, the diffusive capacity saturates even if the fracture length increases. Therefore, we conclude that increasing the fracture length after the saturation point is a waste of resources.

Even though this analysis was performed for a simple model, the  methods we employed in this paper  can be used for more complex problems, involving nonlinear flows in complex fractured geometries, which will be investigated in our future work.

\section{Acknowledgments}
This work was supported by the National Science Foundation grant DMS-1412796.

\section*{References}
\bibliographystyle{plain}
\bibliography{pushpi}
\end{document}